\documentclass[12pt,a4paper]{amsart}
\usepackage{mathrsfs}

\newtheorem{theo+}              {Theorem}           [section]
\newtheorem{prop+}  [theo+]     {Proposition}
\newtheorem{coro+}  [theo+]     {Corollary}
\newtheorem{lemm+}  [theo+]     {Lemma}
\newtheorem{exam+}  [theo+]     {Example}
\newtheorem{rema+}  [theo+]     {Remark}
\newtheorem{defi+}  [theo+]     {Definition}
\newtheorem{clai+}  [theo+]     {Claim}
\newenvironment{theorem}{\begin{theo+}}{\end{theo+}}
\newenvironment{proposition}{\begin{prop+}}{\end{prop+}}
\newenvironment{corollary}{\begin{coro+}}{\end{coro+}}

\usepackage{amsthm}
\theoremstyle{plain} \theoremstyle{remark}
\newtheorem{remark}{Remark}
\newtheorem{example}{Example}

\def \r{\mbox{${\mathbb R}$}}
\def\E{/\kern-1.0em \equiv }

\author{Ze-Ping Wang and Xue-Yi Chen }

\address{ }

\thanks{ School of Mathematical Sciences, Guizhou
Normal University, Guiyang 550025,
People's Republic of China. E-mail:zepingwang@gznu.edu.cn (Ze-Ping Wang is the corresponding author), \\
\indent 
E-mail:xueyichen@gznu.edu.cn (X.-Y. Chen)}
\begin{document}
\title[Biharmonic conformal immersions] {Biharmonic conformal immersions into  a 3-dimensional conformally flat space}
\date {9/4/2024} \subjclass{58E20, 53C12, 53C42} \keywords{Harmonic maps, biharmonic maps, biharmonic conformal immersions, 3-dimensional conformally flat spaces.}
\maketitle

\section*{Abstract}
\begin{quote}
{\footnotesize  Inspired by the work of Ou \cite{Ou1,Ou4}, we study biharmonic conformal immersions  of surfaces into  a  conformally flat 3-space. We first give a characterization of biharmonic conformal immersions of  totally umbilical surfaces into  a generic 3-manifold. As an application,  we 
 give a method to produce  biharmonic conformal immersions into a conformally flat 3-space.  We then use the method to obtain a classification of biharmonic maps in a family of conformal immersions  and construct many examples of biharmonic conformal immersions from a 2-sphere  into a conformal 3-sphere.  
 Our examples include proper biharmonic conformal immersions of a 2-sphere minus a point into a conformal 3-sphere with   nonconstant conformal factor
 and the biharmonic isometric immersion $S^2(\frac{1}{\sqrt{2}})\to S^3$  which was found in \cite{CMO}. 
 Finally,  we  study  biharmonic conformal immersions of  Hopf cylinders  of a Riemannian submersion. }
\end{quote}

\section{Introduction and Preliminaries}
Recall that a  map   $\varphi:(M, g)\to (N,
h)$ between Riemannian manifolds  is called a harmonic map if its tension field  vanishes identically, i.e., $\tau(\varphi)={\rm
Trace}_{g}\nabla {\rm d} \varphi\equiv 0$ (see e.g., \cite{EL1}).
Biharmonic maps are critical
points of the bienergy functional for maps  $\varphi:(M, g)\to (N,
h)$ between Riemannian manifolds:
\begin{equation}\nonumber
E^{2}\left(\varphi,\Omega \right)= \frac{1}{2} {\int}_{\Omega}
\left|\tau(\varphi) \right|^{2}v_g,
\end{equation}
where $\Omega$ is a compact domain of $M$ and
$\tau(\varphi)={\rm Trace}_{g}\nabla {\rm d} \varphi$ is the
tension field of $\varphi$. The Euler–Lagrange equation of the  functional gives
the biharmonic map equation (\cite{Ji}): 
\begin{equation}\label{BT1}\notag
\tau_{2}(\varphi):={\rm
Trace}_{g}(\nabla^{\varphi}\nabla^{\varphi}-\nabla^{\varphi}_{\nabla^{M}})\tau(\varphi)
- {\rm Trace}_{g} R^{N}({\rm d}\varphi, \tau(\varphi)){\rm
d}\varphi =0.
\end{equation}
Here,  $R^{N}$ denotes the curvature operator of $(N, h)$ defined by
$$R^{N}(X,Y)Z=
[\nabla^{N}_{X},\nabla^{N}_{Y}]Z-\nabla^{N}_{[X,Y]}Z.$$
It is clear that any harmonic map is a biharmonic map. So, we are interested  in non-harmonic biharmonic maps which are named  {\bf proper biharmonic maps}.\\

Recall that a submanifold is called a biharmonic submanifold if the isometric immersion that defines the submanifold is a biharmonic map.  The biharmonic equation for a hypersurface in a Riemannian manifold $(N^{m+1}, h)$ defined by an isometric immersion  $\varphi:(M^{m},g)\to (N^{m+1},h)$ can be stated as
\begin{theorem}$($\cite{Ou02}$)$\label{0}
Let $\varphi:M^{m}\to N^{m+1}$ be a hypersurface with mean curvature vector field
$\eta=H\xi$. Then $\varphi$ is biharmonic if and only if:
\begin{equation}\label{BHEq}\notag
\begin{cases}
\Delta H-H |A|^{2}+H{\rm
Ric}^N(\xi,\xi)=0,\\
 2A\,({\rm grad}\,H) +\frac{m}{2} {\rm grad}\, H^2
-2\, H \,({\rm Ric}^N\,(\xi))^{\top}=0,
\end{cases}
\end{equation}
where ${\rm Ric}^N : T_qN\longrightarrow T_qN$ is the Ricci
operator of the ambient space defined by $\langle {\rm Ric}^N\,
(Z), W\rangle={\rm Ric}^N (Z, W)$ and  $A$ denotes the shape operator
of the hypersurface with respect to the unit normal vector field $\xi$.
\end{theorem}

It is  well known that a surface in a 3-sphere $S^3$ defined by the isometric
immersion $S^2(\frac{1}{\sqrt{2}})\to S^3$  is proper biharmonic  ( \cite{CMO}).   For more examples  on biharmonic hypersurfaces, we refer the readers to the book \cite{OC} and the references therein. 

Biharmonic conformal immersions of  hypersurfaces into $(N^{m+1}, h)$ generalize the notion of biharmonic isometric immersions of hypersurfaces (i.e., biharmonic hypersurfaces). More precisely, suppose that  a hypersurface in $(N^{m+1}, h)$ is defined by an isometric immersion $\varphi:(M^{m},g)\to (N^{m+1},h)$   with the induced metric $g=\varphi^*h$. Then,  we call a conformal immersion of the  associated hypersurface into $(N^{m+1}, h)$ is a biharmonic conformal immersion of the  hypersurface 
if there is a function $\lambda : M^m \to(0,+\infty)$ 
such that the conformal immersion $\varphi : (M^m, \bar{g}=\lambda^{-2}g)\to (N^{m+1}, h)$  with the conformal
factor $\lambda$ is a biharmonic map. Note that $g=\varphi^*h=\lambda^2\bar{g}$, and if  $\lambda=1$, then the conformal immersion  is actually an isometric immersion.

In particular,  the biharmonic equation for the conformal immersion of a surface into a  3-dimensional model space can be stated as

\begin{theorem}$($ see e.g., \cite{Ou1, Ou4, Ou5}$)$\label{ff0}
Let  $\varphi:(M^{2},g)\to (N^{3},h)$ be an isometric immersion  with mean curvature vector field
$\eta=H\xi$. Then, the conformal immersion $\varphi:(M^{2},g_0=f^{-1}g)\to (N^{3},h)$ with the conformal factor $\lambda=f^{\frac{1}{2}}$ is biharmonic if and only if:
\begin{equation}\label{ff1}
\begin{cases}
\Delta(fH)-(fH)[|A|^{2}-{\rm
Ric}^N(\xi,\xi)]=0,\\
 A\,({\rm grad}\,(fH)) +(fH)[ {\rm grad}\, H
- \,({\rm Ric}^N\,(\xi))^{\top}]=0,
\end{cases}
\end{equation}
where ${\rm Ric}^N : T_qN\to T_qN$ is the Ricci
operator of the ambient space defined by $\langle {\rm Ric}^N\,
(Z), W\rangle={\rm Ric}^N (Z, W)$ and  $A$ denotes the shape operator
of the surface with respect to the unit normal vector field $\xi$, and 
the operators $\Delta, {\rm grad}$ and $|,| $ are taken with respect to the induced metric $g=\varphi^*h$ on the surface.
\end{theorem}
Note that the induced metric of the surface $g=\varphi^*h=\lambda^2 g_0=fg_0$ and the conformal factor $\lambda=f^{\frac{1}{2}}$.

It is obvious  that  by both Theorem \ref{0} and Theorem \ref{ff0}, if a surface is a minimal surface (i.e., $H=0$) in $(N^3,h)$, then any conformal immersion of the associated  surface  into  $(N^3,h)$  is harmonic and hence biharmonic;  if a surface in $(N^3,h)$ is biharmonic, then a conformal immersion of the associated  surface  into  $(N^3,h)$ with any positive constant conformal factor is  always biharmonic. Such maps are trivially biharmonic conformal immersions. Therefore, we are more interested in those conformal immersions which are proper biharmonic  with nonconstant conformal factor  in this paper.\\

 Motivated by many interesting examples and the rich theory of harmonic and biharmonic immersions  of  surfaces, we study biharmonic conformal immersions  of a Hopf cylinder and a 2-sphere into  a conformally flat 3-space in this paper. There exist many interesting examples and important results of harmonic and biharmonic immersions from   a 2-sphere and a  cylinder into a generic manifold:\\
 
$\bullet$ Chern-Goldberg \cite{CG}: Any harmonic immersion
$\phi:S^2\longrightarrow (N^n,h)$ has to be minimal, or equivalently, a
conformal immersion. 

$\bullet$ Sacks-Uhlenbeck \cite{SU}: Any harmonic map
$\phi:S^2\longrightarrow (N^n,h)$ with $n\ge 3$ has to be a conformal
branched minimal immersion. 

$\bullet$ No a part of the standard sphere $S^2$ can be biharmonically conformally
 immersed into $\r^3$ (see \cite{Ou4}).

$\bullet$  Non-conformal rotationally symmetric map $\varphi:(S^2, dr^2+\sin^2rd\theta^2)\to(N^2,d\rho^2+(\rho^2+C)d\phi^2)$ with $\varphi(r,\theta)=(|\cot\frac{r}{2}|[1+\ln(1+\tan\frac{r}{2})],\theta)$ is a proper biharmonic map defined locally on a 2-sphere (with two singular points), where constant $C>0$ (see \cite{WOY1}). In addition, for $f=\frac{4\tan^2\frac{r}{2} (1+\tan^2\frac{r}{2})(1+2\tan^2\frac{r}{2})^{\frac{3}{2}}}{3\tan^2\frac{r}{2}(1+\tan^2\frac{r}{2})(2+\tan^2\frac{r}{2})+1}$, then non-conformal rotationally symmetric map $\varphi:(S^2, f^{-1}(dr^2+\sin^2rd\theta^2))\to(S^2,d\rho^2+\sin^2\rho d\phi^2)$ with $\varphi(r,\theta)=(\frac{1}{2}\arccos (\sin^2\frac{r}{2}),2\theta)$ from a Riemann 2-sphere to a 2-sphere is a proper biharmonic map defined locally on a Riemann 2-sphere (with two singular points) (see \cite{WOY2}).

$\bullet$  A proper  biharmonic conformal immersion of a developable surface into $\r^3$  exists only in
the case when the surface is a cylinder (see \cite{Ou4, WQ}).

$\bullet$  Consider a circle $\gamma:I\to S^2$ with  radius $\frac{1}{\kappa}$, ($\kappa>1$) and the Hopf cylinder $\varphi: \Sigma=\cup_{s\in I}\pi^{-1}(\gamma(s))\to S^2\times\r$ with the induced metric $g=\varphi^*h$. For some constants $d_1$ and $d_2$, and the function $f=d_1e^{z\sqrt{\kappa^2-1}}+d_1e^{-z\sqrt{\kappa^2-1}}$, then the conformal immersion \;$\varphi:(\Sigma,\bar{ g}=f^{-1}\varphi^*h)\to S^2\times\r$  is a proper biharmonic map (see \cite{Ou4})).  Note that the Hopf cylinder  has constant mean curvature $H=\frac{\kappa}{2}$.

$\bullet$  For some constants $d_1$ and $d_2$, and the function $f=(d_1z+d_2)(1+s^2)^{3/2}$, then the conformal immersion of the cylinder into $\r^3$, i.e.,  $\varphi:(\Sigma, f^{-1}(ds^2+dz^2))\to\r^3$ with $\varphi(s,z)=(\ln(\sqrt{1+s^2}+s)+\ln 4,
 \sqrt{1+s^2},\; z)$,  is a proper biharmonic map (see \cite{WQ}).

For more account on basic examples and properties of biharmonic maps, one refers the readers to \cite{BFO, CH, LiO,Lu,LO,Oua, Ou1,Ou2,Ou3,Ou5} and 
the  references therein.
 
As far as biharmonic maps from $S^2$ (globally defined on a 2-sphere) are
concerned we know nothing more than the biharmonic isometric immersion $S^2(\frac{1}{\sqrt{2}})\to S^3$  (see \cite{CMO})  (or a composition of this with a totally geodesic map from a 3-sphere $S^3$ into another model space ).  So, it would be interesting to know if there are proper biharmonic conformal immersions from $S^2$ or find new examples of such maps.  Knowing from \cite{Ou4, WQ} that  a proper  biharmonic conformal immersion of a developable surface into $\r^3$  exists only in the case where the surface is a cylinder, we want to  classify or construct biharmonic  conformal immersions of  Hopf cylinders of a Riemannian submersion from some 3-manifolds.\\

 In this paper, we study biharmonic conformal immersions  of surfaces into  a  conformally flat 3-space. We first give a characterization of biharmonic conformal immersions of a totally umbilical surface into  a generic 3-manifold (Theorem \ref{FS1}). As an application,  we  give a method to produce  biharmonic conformal immersions into a conformally flat 3-space (Proposition \ref{PSSL0}).  We then use the method to obtain a classification of biharmonic maps in a family of conformal immersions  and construct a family of infinitely many proper biharmonic conformal immersions from a 2-sphere minus a point into a conformal 3-sphere with nonconstant conformal factor and one singular point at the north pole  (Theorem \ref{SSL}).  
 We are able to  give the example of   biharmonic isometric immersion $S^2(\frac{1}{\sqrt{2}})\to S^3$  which was found in \cite{CMO} with a different method  (Proposition \ref{ISS}). 
 Finally, we  derive biharmonic equation for conformal immersions of  Hopf cylinders  of a Riemannian submersion (Proposition \ref{Hopf0}). We then apply the equation to give a  classification of biharmonic conformal immersions  of the Hopf cylinders into BCV 3-spaces which include conformally flat spaces $S^3$, $\r^3$, $S^2\times\r$ and $H^2\times\r$ (Theorem \ref{hopf} and Corollary \ref{chopf}).

\section{Biharmonic  conformal immersions of  totally umbilical surfaces into a 3-dimensional conformally space}

In this section, we give a characterization  of biharmonic  conformal immersions of  totally umbilical surfaces into a 3-dimensional model space. As  applications, 
we give a method to produce proper biharmonic  conformal immersions which are constructed by starting with a totally umbilical surface in a conformally ﬂat space and then performing a suitable conformal change of the conformally ﬂat metric into another conformally ﬂat metric.  We then use the method to  study  biharmonicity  of  a family of  conformal immersions  and  construct  proper biharmonic conformal immersions from  $S^2$ into a conformal flat 3-sphere. We  also provide  many examples of   biharmonic conformal immersions from $\r^2$ or $H^2(-1)$.

\subsection{ A characterization of biharmonic  conformal immersions of   totally umbilical surfaces into a  Riemannian 3-manifold}

\begin{theorem}\label{FS1}
Let  $\varphi:(M^{2},g)\to (N^{3},h)$  be a totally umbilical  surface  with mean curvature vector field
$\eta=H\xi$. Then,  the conformal immersion $\varphi:(M^{2},f^{-1}g)\to (N^{3},h)$ is biharmonic  if and only if one of the following  cases happens:\\
$(i)$: the surface  is  totally geodesic; \\
$(ii)$  the surface has nonzero mean curvature $H$,   and $f$ and $H$ satisfy 
\begin{equation}\label{fs1}
\begin{array}{lll}
({\rm Ric}^N\,(\xi))^{\top}={\rm grad}_gH,\;\;\;
{\rm Ric}^N(\xi,\xi)=|A|^2=2H^2,\;\;{\rm grad}_g(fH)=0.
\end{array}
\end{equation}
In this case,  we have  nonzero function $f=c|H|^{-1} $ on the surface,
where $c>0$ is a constant. Moreover, by the totally umbilical property of the surface, the first equation of (\ref{fs1})  holds naturally. 
\end{theorem}

\begin{proof}
One can choose an orthonormal frame $\{e_i,\;
\xi,\;i=1,2\}$ of $(N^3,h)$ adapted to the surface  with $\xi$ being the unit normal vector field.
If the surface is totally umbilical, then we have $A(e_i) = H e_i$ for $i=1,2$,  $|A|^2 = 2H^2$ and $ A({\rm grad}_g H)=H{\rm
grad}_g\, H.$

A straightforward computation  gives
\begin{eqnarray}\label{Cod1}
\sum_{i=1}^{2}\langle R^N(e_i,e_j)e_i,\xi\rangle =\sum_{i=1}^{2}R^N(\xi,
e_i,e_i, e_j) =-{\rm Ric}^N(e_j,\xi) .
\end{eqnarray}
Note that $\{e_i\}$, $i=1,2$,  are principal
directions with principal curvature $H$ one can compute that
\begin{eqnarray}\label{Cod2}
(\nabla_{e_i}h)(e_j,e_i)&=&{e_i}(h(e_j,e_i))-h(\nabla_{e_i}e_j,e_i)-h(\nabla_{e_i}e_i,e_j)\\\notag
&=&-h(e_i,e_i)\langle\nabla_{e_i}e_j,e_i\rangle-h(e_j,e_j)\langle\nabla_{e_i}e_i,e_j\rangle\\\notag
&=&-H(\langle\nabla_{e_i}e_j,e_i\rangle+\langle\nabla_{e_i}e_i,e_j\rangle)=0,
\end{eqnarray}
and
\begin{equation}\label{Cod3}
\begin{array}{lll}
(\nabla_{e_j}h)(e_i,e_i)={e_j}(h(e_i,e_i))-h(\nabla_{e_j}e_i,e_i)-h(\nabla_{e_j}e_i,e_i)\\
=e_j(H)-2h(e_i,e_i)\langle\nabla_{e_j}e_i,e_i\rangle
=e_j(H),\;\;\;\;for\;j=1,2.
\end{array}
\end{equation}
Substracting (\ref{Cod2}) from (\ref{Cod3}) and using (\ref{Cod1}) and the Codazzi equation
 for a surface we have
\begin{eqnarray}\label{Cod4}
e_j(H)={\rm R}^N(\xi, e_i,e_j, e_i).
\end{eqnarray}
Summing both sides of (\ref{Cod4}) over $i$ from 1 to 2 yields
\begin{equation}\label{Cod5}
e_j(H)={\rm Ric}^N(e_j,\xi),\; j=1, 2,\;i.e.,\;{\rm grad}_gH =({\rm Ric}^N\,(\xi))^{\top}.
\end{equation}
It is a fact that  (\ref{Cod5})  holds naturally since the totally umbilical property of the surface. \\

Notice  that if $H=0$, then the surface is totally geodesic and hence biharmonic. From now on, consider that $H\neq0$. 
Substituting  (\ref{Cod5}) into the 2nd equation of
(\ref{ff1}) and using the totally umbilical property of the surface  yields
\begin{equation}\label{Cod6}
{\rm grad}_g(fH) =0,
\end{equation}
which implies that the function $f|H|=c$ is a positive constant on the surface.  We substitute this into  the 1st equation of
(\ref{ff1}) to  obtain 
\begin{equation}\label{Cod7}
{\rm Ric}^N(\xi,\xi)=|A|^2=2H^2.
\end{equation}
Clearly, one observes  that  if $H={\rm constant}\neq0$, then $f$  has to be a constant and the surface is actually biharmonic.

Summarizing the above results we obtain the theorem.
\end{proof}
From Statement (ii) in Theorem \ref{FS1},  the Ricci curvature ${\rm Ric}^N(\xi,\xi)=|A|^2=2H^2>0$ on the surface,  so we have
\begin{corollary}\label{TGC}
Any biharmonic conformal immersion  of a totally umbilical  surface into  a nonpositively curved manifold  $(N^{3},h)$ is harmonic. In this case, the surface in  $(N^{3},h)$ has to be a totally geodesic surface.
\end{corollary}

\begin{remark}
 Applying Theorem \ref{FS1} and using the known  results (see e.g., \cite{CMO,CH,OC}), we can conclude that if a  conformal immersion of a totally umbilical surface into $\r^3$, $H^3$  or $S^3$ is biharmonic, then the associate surface has to be either a totally geodesic surface or  $S^2(\frac{1}{\sqrt{2}})$ in $S^3$, which  was obtained in \cite{Ou5}.
\end{remark}

\subsection{A method to  construct biharmonic  conformal immersions into a 3-dimensional conformally space.\\}

Now, we are ready to give the  method to construct  proper biharmonic  conformal immersions.

\begin{proposition}\label{PSSL0}
 Consider  an isometric immersion  $(\r^2,g_0=\varphi^*h_0) \to (\r^3, h_0 =F^{-2}(dx^2+dy^2+dz^2))$ with $\varphi(x, y) = (x,y,a_1x+a_2y+a_3)$, where $a_1,a_2$ and $a_3$ are constants. Then, the conformal immersion $\varphi: (\r^{2}, f^{-1}\varphi^*h)\to(\r^{3}, h=\beta^{-2}h_0)$   is proper biharmonic if and only if 
 $F\beta $ satisfies the following 
\begin{equation}\label{PSS}
\begin{array}{lll}
\frac{1+2a_1^2+a_2^2}{1+a_1^2+a_2^2}[F\beta(F\beta)_{xx}-2((F\beta)_x)^2]+\frac{1+a_1^2+2a_2^2}{1+a_1^2+a_2^2}[F\beta(F\beta)_{yy}-2((F\beta)_y)^2]\\
+\frac{2+a_1^2+a_2^2}{1+a_1^2+a_2^2}[F\beta(F\beta)_{zz}-2((F\beta)_z)^2]
+\frac{2a_1a_2}{1+a_1^2+a_2^2}[F\beta(F\beta)_{xy}-2(F\beta)_x(F\beta)_y]\\

-\frac{2a_1}{1+a_1^2+a_2^2}[F\beta(F\beta)_{xz}-2(F\beta)_x(F\beta)_z]
-\frac{2a_2}{1+a_1^2+a_2^2}[F\beta(F\beta)_{yz}-2(F\beta)_y(F\beta)_z]=0.
\end{array}
\end{equation}
when restricting to the surface $z=a_1x+a_2y+a_3$, and 
\begin{equation}\label{PSS0}
f=\frac{c\sqrt{1+a_1^2+a_2^2}}{|-a_1(F\beta)_ x-a_2(F\beta)_y+(F\beta)_ z|}\circ\varphi,
\end{equation}
 where constant  $c>0$. Furthermore, the conformal factor $\lambda=f^{\frac{1}{2}}$ and the surface  $(\r^2,\varphi^*h)$ has the mean curvature $H=\frac{-a_1(F\beta)_ x-a_2(F\beta)_y+(F\beta)_ z}{\sqrt{1+a_1^2+a_2^2}}$.
\end{proposition}
\begin{proof}
First of all, we give the following claim:\\
{\bf Claim:} The surface $\varphi: (\r^{2},g=\varphi^*h_0)\to(\r^{3}, h_0=F^{-2}(dx^2+dy^2+dz^2))$ with $\varphi(x, y) = (x,y,a_1x+a_2y+a_3)$    is totally umbilical with the mean curvature $H_0=\frac{-a_1F_ x-a_2F_y+F_ z}{\sqrt{1+a_1^2+a_2^2}}$, where $F:\r^3\to(0,+\infty)$ and $a_i$ are constants for $i=1,2,3$.\\

{\bf Proof of Claim}.
We can easily check that the surface  $\varphi : (\r^2,\varphi^*h^0) \to (\r^{3}, h^0 = 
 dx^2+dy^2+dz^2)$ with $\varphi(x, y) = (x,y,a_1x+a_2y+a_3)$ is  totally geodesic with the unit normal vector field $\xi^0=\frac{-a_1\frac{\partial}{\partial x}-a_2\frac{\partial}{\partial y}+\frac{\partial }{\partial z}}{\sqrt{1+a_1^2+a_2^2}}$ and the mean curvature $H^0=0$.   It follows from the well-known result (see e.g., \cite{CL,LiO,Ou2}) that $\varphi: (\r^{2},g_0=\varphi^*h_0)\to(\r^{3}, h_0=F^{-2}(dx^2+dy^2+dz^2))$, $\varphi(x, y) = (x,y,a_1x+a_2y+a_3)$    is totally umbilical with the unit normal  vector field  $\xi_0=F\xi^0$ and the  mean curvature $H_0=F H^0+\xi^0 F=\xi^0 F$. From which we get the claim.\\
 
 On the other hand, it follows from   Claim that the surface $\varphi: (\r^{2},g=\varphi^*h)\to(\r^{3}, h=\beta^{-2}h_0)$, $\varphi(x, y) = (x,y,a_1x+a_2y+a_3)$   is  totally umbilical with the unit normal  vector field  $\xi=\beta\xi_0$ and the  mean curvature
$H=\frac{-a_1(F\beta)_ x-a_2(F\beta)_y+(F\beta)_ z}{\sqrt{1+a_1^2+a_2^2}}$.  A straightforward computation gives (see e.g., \cite{LiO})
\begin{equation}\label{c1}\notag
\begin{array}{lll}
{\rm Ric}(\xi,\xi)=\Delta_{h}\ln (F\beta)+[{\rm Hess}_{h}(\ln (F\beta)(\xi,\xi)-
(\xi\ln (F\beta))^2+|{\rm grad}_{h}\ln (F\beta)|_{h}^2]\\

=(F\beta)\Delta_{h^0}(F\beta)-2|{\rm grad}_{h^0}(F\beta)|^2_{h^0}+F\beta{\rm Hess}_{h^0}(F\beta)(\xi^0,\xi^0)\\

=\frac{1+2a_1^2+a_2^2}{1+a_1^2+a_2^2}F\beta(F\beta)_{xx}-2((F\beta)_x)^2+\frac{1+a_1^2+2a_2^2}{1+a_1^2+a_2^2}F\beta(F\beta)_{yy}-2((F\beta)_y)^2\\
+\frac{2+a_1^2+a_2^2}{1+a_1^2+a_2^2}F\beta(F\beta)_{zz}-2((F\beta)_z)^2+\frac{2a_1a_2F\beta(F\beta)_{xy}}{1+a_1^2+a_2^2}-\frac{2a_1F\beta(F\beta)_{xz}}{1+a_1^2+a_2^2}
-\frac{2a_2F\beta(F\beta)_{yz}}{1+a_1^2+a_2^2}.
\end{array}
\end{equation}

 Then, it follows from Theorem \ref{FS1} that
 the conformal immersion $\varphi: (\r^{2}, f^{-1}\varphi^*h)\to(\r^{3}, h=\beta^{-2}h_0)$   is proper biharmonic if and only if 
$f=c|H|^{-1}=\frac{c\sqrt{1+a_1^2+a_2^2}}{|-a_1(F\beta)_ x-a_2(F\beta)_y+(F\beta)_ z|}$, where constant $c>0$, and $ {\rm Ric}(\xi,\xi)=2H^2$
 holds on  the surface  implying that (\ref{SS}) holds on the surface.

From which the proposition follows.
\end{proof}

Applying Proposition \ref{PSSL0}, we immediately have the following corollaries which can be applied to construct  infinitely many examples of proper biharmonic conformal immersions. 
\begin{corollary}\label{SSL0}
 Consider  an isometric immersion  $(\r^2,g_0=\varphi^*h_0) \to (\r^3, h_0 =F^{-2}(dx^2+dy^2+dz^2))$ with $\varphi(x, y) = (x,y,a_3)$, where $a_3$ is a constant. Then, the conformal immersion $\varphi: (\r^{2}, f^{-1}\varphi^*h)\to(\r^{3}, h=\beta^{-2}h_0)$   is proper biharmonic if and only if 
 $F\beta $ satisfies the following 
\begin{equation}\label{SS}
F\beta(F\beta)_{xx}-2((F\beta)_x)^2+F\beta(F\beta)_{yy}-2((F\beta)_y)^2+2F\beta(F\beta)_{zz}-4((F\beta)_z)^2=0
\end{equation}
when restricting to the surface $z=a_3$, and 
\begin{equation}\label{SS0}
f=\frac{c}{|(F\beta )_z|}\circ\varphi,
\end{equation}
 where constant  $c>0$. Moreover, the conformal factor $\lambda=f^{\frac{1}{2}}$ and  the surface  $(\r^2,\varphi^*h)$ has the mean curvature $H=(F\beta )_z$.
\end{corollary}
\begin{proof}
Taking  $a_1=a_2=0$ in Proposition \ref{PSSL0}, the corollary follows.
\end{proof}

\begin{corollary}\label{PQ1}
For constants $a_1, a_2$ and $a_3$,
then the conformal immersion  $\varphi : (\r^2,f^{-1}\varphi^*h)  \to (\r^3, h = \beta^{-2}(z)(
 dx^2+dy^2+dz^2))$ with $\varphi(x, y) = (x,y,a_1x+a_2y+a_3)$ is proper biharmonic if and only if 
  $\beta(z)\neq{\rm constant}$  satisfies the following 
\begin{equation}\label{pq0}
\beta\beta''-2\beta'^2=0
\end{equation}
when restricting to the surface $z=a_1x+a_2y+a_3$, and 
\begin{equation}\label{pq1}\notag
f=\frac{c \;\sqrt{1+a_1^2+a_2^2}}{|\beta'(z)|}\circ\varphi,
\end{equation}
 where constant  $c>0$. Moreover, the conformal factor $\lambda=f^{\frac{1}{2}}$ and the surface  $(\r^2,\varphi^*h)$ has the mean curvature $H=\frac{\beta'(z)}{\sqrt{1+a_1^2+a_2^2}}$. 
 
\end{corollary}
\begin{proof}
Applying Proposition \ref{PSSL0} with $F=1$ and $\beta=\beta(z)$ (depending only on $z$),   we immediately get the corollary.
\end{proof}

Taking $a_1=a_2=1$, $a_3=0$, $c=\frac{\sqrt{3}}{3}$ and $\beta(z)=\frac{1}{z}$ in Corollary \ref{PQ1}, we immediately have the example as follows.
\begin{example}\label{pq1}
The conformal immersion  $\varphi : (\r^2_{+},2(dx^2+dxdy+dy^2))  \to (\r^3_{+},  z^{2}(
 dx^2+dy^2+dz^2))$,  $\varphi(x, y) = (x,y,x+y)$ is proper biharmonic with conformal factor $f^{\frac{1}{2}}=x+y$, where $\r^2_{+}=\{(x,y)\in\r^2, x,y>0\}$ and  $\r^3_{+}=\{(x,y,z)\in\r^3, x,y,z>0\}$. Note that the surface $(\r^2_{+},\varphi^*h)$  has nonconstant mean curvature $H=-\frac{1}{\sqrt{3}(x+y)^2}$ and  the domain surface $(\r^2_{+},2(dx^2+dxdy+dy^2))$  is flat. 
\end{example}

\subsection{Biharmonic conformal immersions of a 2-sphere into a conformal 3-sphere.\\}

Let us  now consider a totally geodesic immersion defined by
\begin{equation}\label{TG}
\varphi: (S^2, g_0)\to (S^3,h_0),\;\varphi (u)=(u, 0), \;\forall u\in S^2\subset \r^3
\end{equation}
from the standard 2-sphere into the standard 3-sphere.

It would like to be interesting to know if there is  a nonconstant positive function $\beta $ on the target sphere $ S^3$ so that the conformal immersion  $\varphi: (S^2, g_0)\to (S^3, \beta^{-2} h_0)$ from the standard 2-sphere into the conformal 3-sphere is a proper biharmonic map. Note that the conformal factor $\lambda=f^{\frac{1}{2}}=\beta^{-1}\circ\varphi$.

For the calculations, it is convenient to use local coordinates on  the domain and target spheres, respectively. 
With respect to local coordinates,  the map (\ref{TG}) can be described as 
\begin{equation}\label{TG1}\notag
\begin{array}{lll}
\varphi :(S^2\backslash\{N\},g_0=\frac{4(dx^2+dy^2)}{(1+x^2+y^2)^2}) \to(S^3\backslash\{N'\}, h_0 =\frac{4(dx^2+dy^2+dz^2)}{(1+x^2+y^2+z^2)^2}),\\
\varphi(x, y) = (x,y,0),
\end{array}
\end{equation}
where $\{N\}$ and $\{N'\}$ denote the noth poles on the domain and target spheres, respectively.\\
We shall seek   a positive function $\beta $ on the target sphere $ S^3$ so that the conformal immersion given by

\begin{equation}\label{TG2}\notag
\begin{array}{lll}
\varphi :(S^2\backslash\{N\},g_0) \to(S^3\backslash\{N'\}, h=\beta^{-2}h_0),\;
\varphi(x, y) = (x,y,0)
\end{array}
\end{equation}
is proper biharmonic.

In what follows,  we defines the function  $\Phi_k(x,y,z)$  by 
\begin{equation}\label{GG3}
\begin{array}{lll}
\Phi_k(x,y,z)=\frac{2r^3}{-2rz-2(r^2+z^2)\arctan\frac{z}{r}+k\;r^3(r^2+z^2)}
\end{array}
\end{equation}
where $k$ is a constant and  $r=\sqrt{1+x^2+y^2}$.

Note that for $k\geq6$,  the function $-2rz-2(r^2+z^2)\arctan\frac{z}{r}+k\;r^3(r^2+z^2)\geq  -2rz+(r^2+z^2)(6r^3-\pi)> (r-z)^2+(r^2+z^2)\geq1$ and hence $\Phi_k(x,y,z)>0$ when  $k\geq6$.\\

Our next theorem gives a classification of biharmonic maps in a family of conformal immersions of the standard
2-sphere into the conformal 3-sphere and  a family of proper biharmonic conformal immersions   defined on $S^2\backslash\{N\}$.

\begin{theorem}\label{SSL}
(i)  Any proper biharmonic conformal immersion $\varphi: (S^2, g_0)\to (S^3,h=\beta^{-2}h_0)$ from the standard 2-sphere into the  conformal  3-sphere with \;$\varphi (u)=(u, 0), \;\forall u\in S^2\subset \r^3$, is actually defined on the 2-sphere with at least one point  deleted.\\

(ii) Let $\psi_w(x,y,z)$ be defined by (\ref{gg1}) and $w=w(x,y)$ be a positive harmonic function so that $\psi_w(x,y,z)>0$ on $\r^3$. For  $\beta=\frac{2\psi_w(x,y,z)}{1+x^2+y^2+z^2}$,  
then the conformal immersion $\varphi : (S^2\backslash\{N\},g_0=\frac{4(dx^2+dy^2)}{(1+x^2+y^2)^2}) \to(S^3\backslash\{N'\}, h=\beta^{-2}h_0=\beta^{-2}\;\frac{4(dx^2+dy^2+dz^2)}{(1+x^2+y^2+z^2)^2})$  with $\varphi(x, y) = (x,y,0)$ into the conformal 3-sphere  is proper biharmonic  for the conformal factor $f^{\frac{1}{2}}=\frac{w(x,y)(1+x^2+y^2)}{2}$.\\

(iii) For a family of   positive functions $\beta=\Phi_k(x,y,z)$ defined by (\ref{GG3}) with $k\geq6$, then the conformal immersion $\varphi : (S^2\backslash\{N\},g_0=\frac{4(dx^2+dy^2)}{(1+x^2+y^2)^2}) \to(S^3\backslash\{N'\}, h=\beta^{-2}h_0=\beta^{-2}\;\frac{4(dx^2+dy^2+dz^2)}{(1+x^2+y^2+z^2)^2})$  with $\varphi(x, y) = (x,y,0)$ into the conformal 3-sphere  is proper biharmonic for the conformal factor $f^{\frac{1}{2}}=\frac{k(1+x^2+y^2)}{2}$. 
\end{theorem}
\begin{proof}
First of all,  one can easily check that  $\varphi : (\r^2,g_0=\frac{4(dx^2+dy^2)}{(1+x^2+y^2)^2}) \to (\r^3, h_0 =F^{-2}(dx^2+dy^2+dz^2)=\frac{4(dx^2+dy^2+dz^2)}{(1+x^2+y^2+z^2)^2})$, $\varphi(x, y) = (x,y,0)$  is a totally umbilical surface  with the unit normal vector field\; $\xi_0=F\frac{\partial}{\partial z}$  and the mean curvature $H_0|_{z=0}=F_z=z|_{z=0}=0$.  
Then, by Corollary \ref{SSL0}, we can conclude that the totally umbilical surface $\varphi: (\r^{2},g=\varphi^*h)\to(\r^{3}, h=\beta^{-2}h_0)$, $\varphi(x, y) = (x,y,0)$  has the unit normal  vector field  $\xi=F\beta\frac{\partial }{\partial z}$ and the  mean curvature $H=(F\beta)_z$.

Secondly, we will seek the functions $f$ and $\beta$ such that $f^{-1}\varphi^*h=g_0$, i.e., $\frac{c(F\beta )_z}{(F\beta)^2}=\pm F^{-2}$.
We now consider the following $PDE$ 
\begin{equation}\label{gg0}
\begin{array}{lll}
\frac{(F\beta )_z}{(F\beta)^2}= F^{-2},
\end{array}
\end{equation}
where  $F=\frac{1+x^2+y^2+z^2}{2}$. 

The equation (\ref{gg0}) has the general solutions given by
\begin{equation}\label{gg1}
\begin{array}{lll}
F\beta=\psi_{w}(x,y,z)=\frac{r^3(r^2+z^2)}{-2rz-2(r^2+z^2)\arctan\frac{z}{r}+w(x,y)\;r^3(r^2+z^2)},
\end{array}
\end{equation}
where $r=\sqrt{1+x^2+y^2}$ and $w=w(x,y)$ is  a nonzero differentiable function.

We shall show that $w(x,y)$ is  nonzero harmonic function on $\r^2$, i.e., $w_{xx}+w_{yy}=0$.
In fact, substituting $F\beta=\psi_w(x,y,z)$ into the left-hand side of (\ref{SS}) and restricting  to the surface $z=0$, one can conclude that
\begin{equation}\label{gg2}
\begin{array}{lll}
\{F\beta(F\beta)_{xx}-2((F\beta)_x)^2+F\beta(F\beta)_{yy}-2((F\beta)_y)^2+2F\beta(F\beta)_{zz}-4((F\beta)_z)^2\}|_{z=0}\\
=\{\psi_w(\psi_w)_{xx}-2(\psi_w)_x^2+\psi_w(\psi_w)_{yy}-2(\psi_w)_y^2+2\psi_w(\psi_w)_{zz}-4(\psi_w)_z^2\}|_{z=0}
=-\frac{w_{xx}+w_{yy}}{w^3}.
\end{array}
\end{equation}
Using this and (\ref{SS}), it follows  that $w_{xx}+w_{yy}=0$ and hence $w(x,y)$ is  nonzero harmonic function on $\r^2$.
This implies that in this case $\beta=\frac{2\psi_w(x,y,z)}{1+x^2+y^2+z^2}$ and $f^{\frac{1}{2}}=((F\beta)_z|_{z=0})^{-\frac{1}{2}}=((F\beta)^2F^{-2}|_{z=0})^{-\frac{1}{2}}=\beta^{-1} \circ\varphi=\frac{w(x,y)(1+x^2+y^2)}{2}$ and hence $w(x,y)>0$.\\

 Thirdly, we show that   $f^{\frac{1}{2}}=\frac{w(x,y)(1+x^2+y^2)}{2}$ is locally  defined on $S^2$. If otherwise, i.e., $f^{\frac{1}{2}}=\frac{w(x,y)(1+x^2+y^2)}{2}$
 is globally defined on $S^2$  $( S^2\equiv \r^2\cup\{\infty\})$, 
  this implies that the positive function $w(x,y)$ is a nonconstant, bounded  function on $S^2$ (i.e., $\r^2\cup\{\infty\})$. It follows from  Liouville' theorem that a bounded harmonic function 
  on Euclidean space is a constant, then the $w(x,y)$ is a constant, a contradiction. Therefore,  it follows that the conformal factor $\lambda=f^{\frac{1}{2}}$ has at least one singular point at the north pole.\\

Finally, we show that there  must have a positive harmonic function $w(x,y)$ on $\r^2$ so that either of  the two functions $ f^{\frac{1}{2}}$ and  $\beta$  has the only singular point at the north  pole. Note that  we already mentioneded that for constant $k\geq6$, any function $\Phi_k(x,y,z)$ defined by (\ref{GG3}) is  a positive function. Then,  for constant harmonic function  $w=w(x,y)=k\geq6$, we immediately obtain the two positive functions $\beta=\frac{\psi_{k}(x,y,z)}{F}=\Phi_k(x,y,z)$ on $\r^3$ and $f^{\frac{1}{2}}=\frac{k(1+x^2+y^2)}{2}$ on $\r^2$, each of which has the only singular point at the north  pole. 

 Summarizing all results above,  we get the theorem.
\end{proof}

\begin{remark}
 (1) We would like to point out that  none of the locally defined proper biharmonic conformal immersions from $S^2$ given in Theorem \ref{SSL} can be extended to a global
map  $\varphi : (S^2,g_0) \to(S^3, h=\beta^{-2}h_0)$ since each of them  as conformal immersion has at least one  singular point at the north pole. In particular,  we see that each of  biharmonic conformal immersions in Statement (iii) of Theorem \ref{SSL} has the only  one  singular point at the north pole. \\
It is interesting to compare the examples of proper biharmonic conformal immersions provided by Statement (iii) of Theorem \ref{SSL}  with the  
well-known theorem  by Sacks-Uhlenbeck in \cite{SU}  stating that any harmonic map
$\phi:S^2\longrightarrow (N^n,h)$ with $n\ge 3$ has to be a conformal minimal immersion  away from points where the differential of the map vanishes.\\

(2) By Corollary \ref{TGC},  the conformal 3-sphere $(S^3, h=\beta^{-2}h_0)$ in Theorem \ref{SSL}  cannot be  a nonpositively curved manifold.

\end{remark}

We now give a description of the  biharmonic isometric  immersion $S^2(\frac{1}{\sqrt{2}})\to S^3$  by using
the local coordinates.

\begin{proposition}\label{ISS}
A (part of) sphere $S^2(\frac{1}{\sqrt{2}})$ in $S^3$ is proper biharmonic. Furthermore,  the isometric immersion can be locally expressed as
\begin{equation}\label{ss}
\begin{array}{lll}
\varphi : \left(S^2(\frac{1}{\sqrt{2}}),\frac{4(dx^2+dy^2)}{(2+x^2+y^2)^2}\right) \to\left(S^3, \frac{4(dx^2+dy^2+dz^2)}{(1+x^2+y^2+z^2)^2}\right),\;
 \varphi(x, y) = (x,y,1).
\end{array}
\end{equation}
\end{proposition}
\begin{proof}
Applying Corollary \ref{SSL0} with $F=\frac{1+x^2+y^2+z^2}{2}$ and $\beta=1$, and  using an argument similar to Theorem \ref{SSL}, we can conclude that 
the surface  defined by the isometric immersion 
\begin{equation}\label{ss1}
\begin{array}{lll}
\varphi : \left(S^2(\frac{1}{\sqrt{2}})\backslash\{N\},\frac{4(dx^2+dy^2)}{(2+x^2+y^2)^2}\right) \to\left(S^3\backslash\{N'\}, \frac{4(dx^2+dy^2+dz^2)}{(1+x^2+y^2+z^2)^2}\right),\;
 \varphi(x, y) = (x,y,1)
\end{array}
\end{equation}
is proper biharmonic. Note that  the map $\varphi$ as the isometric immersion  can be extended smoothly to  $S^2(\frac{1}{\sqrt{2}})$.
 From which , we obtain the proposition.

\end{proof}

\begin{remark}
We remark that  the authors in \cite{CMO} showed that a (part of) sphere $S^2(\frac{1}{\sqrt{2}})$ in $S^3$ is proper biharmonic with a diﬀerent method.  Interestingly,  our Proposition \ref{ISS} give a description of the  above proper biharmonic map in the local coordinates.
\end{remark}

At the end of this section, by using an argument similar to Theorem \ref{SSL},  we can construct a family of  biharmonic conformal immersions from $H^2(-1)$ as follows.
\begin{proposition}\label{HSSL}
 For $\rho=\sqrt{1-x^2-y^2}$ and $\beta=\frac{2\rho^3(\rho^2-z^2)}{-2\rho z-2(\rho^2-z^2)\arctan\frac{z}{\rho}+w(x,y)\;\rho^3(\rho^2-z^2)}$, where  $w=w(x,y)$ is a positive harmonic function on $\r^2$ so that $\beta>0$, locally, then the conformal immersion $\varphi : (H^2(-1),g_0=\frac{4(dx^2+dy^2)}{(1-x^2-y^2)^2}) \to(H^3(-1), h=\beta^{-2}\;\frac{4(dx^2+dy^2+dz^2)}{(1-x^2-y^2-z^2)^2})$  with $\varphi(x, y) = (x,y,0)$ from a hyperbolic 2-space into the conformal hyperbolic 3-space is proper biharmonic  for the conformal factor $f^{\frac{1}{2}}=\frac{w(x,y)(1-x^2-y^2)}{2}$.\\

\end{proposition}

\section{Biharmonic  conformal immersions of  Hopf cylinders of a Riemannian submersion}

Let $\pi:(N^3,h)\to (B^2,h_1)$ be a Riemannian submersion with totally geodesic fibres.
Let $\gamma:I\to(B^2, h_1)$  be
 an immersed regular curve parametrized by arclength  and  $ \tilde{\gamma}: I \to (N^3,h)$ be a
horizontal lift of $\gamma$ with the geodesic curvatrue $\kappa$ and the torsion $\tau$. 
 Then $\Sigma=\cup_{s\in I}\pi^{-1}(\gamma(s))$ is a surface  called a {\bf Hopf cylinder} in $N^3$, which
can be also viewed as a disjoint union of all horizontal lifts of the curve $\gamma$.
 
Let $\{X,V,\xi\}$ be an orthonormal frame adapted  to the cylinder with $\xi$ being unit normal vector filed.  We denote by $A$   the shape operator of the surface with
respect to $\xi$.  With respect to the  frame $\{X,V,\xi\}$,  we can compute the mean curvature,  the second fundamental form  and other terms that appear in the equation (\ref{ff1})  given by (see e.g., \cite{Ou4})
 \begin{eqnarray}\label{hf1}
&& A(X)=-\langle\nabla_{X}\xi,X\rangle X-\langle\nabla_{X}\xi,V\rangle V=\kappa X-\tau V,\\\notag
&& A(V)=-\langle\nabla_{V}\xi,X\rangle X-\langle\nabla_{V}\xi,V\rangle V=-\tau X;\\\notag
&& B(X,X)=\langle A(X),X\rangle=\kappa,\; B(X,V)=\langle A(X),V\rangle=-\tau,\\\notag
&& B(V,X)=\langle A(V),X\rangle=-\tau,\; B(V,V)=\langle A(V),V\rangle=0,\\\notag
\end{eqnarray}
and
\begin{eqnarray}\label{hf2}
&& H=\frac{1}{2}(B(X,X)+ B(V,V))=\frac{\kappa}{2},\\\notag
&&A({\rm grad}H)=A(X(\frac{\kappa}{2})X+V(\frac{\kappa}{2})V)=X(\frac{\kappa}{2})A(X)=\frac{\kappa'}{2}(\kappa X-\tau V);\\\notag
&&\Delta H=XX(H)+VV(H)-\nabla_{X}X(H)-\nabla_{V}V(H)=\frac{\kappa''}{2};\\\notag
&&|A|^2=(B(X,X))^2+(B(X,V))^2+(B(V,X))^2+(B(V,V))^2=\kappa^2+2\tau^2.\\\notag
\end{eqnarray}

It is  well know that the Hopf cylinder can be viewed as an isometric immersion $\varphi: \Sigma \to(N^3,h)$ with the induced metric $g=\varphi^*h$. Furthermore, the cylinder can be parametrized by $\varphi(s,z)=(\gamma(s),z)\subset(N^3,h)$.\\

Summarizing the above argument, we  give the biharmonic equation for  conformal immersions of the Hopf cylinders of a Riemannian submersion from   $(N^3,h)$.
 \begin{proposition}\label{Hopf0}
A conformal immersion of  the Hopf cylinder  $\varphi:(\Sigma,\;\bar{ g}=f^{-1}\varphi^*h)\to (N^3,h)$ is biharmonic if and only if:

\begin{equation}\label{Hopf}
\begin{cases}
\kappa''-\kappa^3-2\kappa \tau^2+\kappa{\rm Ric}^N(\xi,\xi)+2\kappa' X(\ln f)\\+\kappa\{XX(\ln f)+VV(\ln f)+(X\ln f)^2+(V\ln f)^2\}=0,\\
3\kappa\kappa'-2\kappa{\rm Ric}^N(\xi,X)+2\kappa^2X(\ln f)-2\kappa \tau V(\ln f)=0,\\
\kappa'\tau+\kappa{\rm Ric}^N(\xi,V)+\kappa\tau X(\ln f)=0.
\end{cases}
\end{equation}
\end{proposition}
\begin{proof}
We substitute (\ref{hf1}) and (\ref{hf2}) into biharmonic  equation (\ref{ff1})  and simplify the resulting equation to obtain (\ref{Hopf}). 
\end{proof}

Recall that we call the 3-dimensional Riemannian manifold
$$M^3{(m,\;l)}=(\r^{3},h=\frac{dx^2+dy^2}{[1+m(x^2+y^2)]^2}+[dz+\frac{l}{2}\frac{y
dx-x dy}{1+m(x^2+y^2)}]^2)$$
 3-diemnsional  Bianchi-Cartan-Vranceeanu  spaces (3-dimensional BCV spaces for
short), which include   conformally flat spaces $\r^3$, $S^3$, $S^2\times\r$ and $H^2\times\r$ besides the  model spaces $Nil$, $\widetilde{SL}(2,\r)$ and $SU(2)$.
A natual orthonormal frame on $M^3{(m,\;l)}$ is given by 
\begin{equation}\label{X190}\notag 
E_{1}=F\frac{\partial}{\partial
x}-\frac{ly}{2}\frac{\partial}{\partial z},\;E_{2}=F
\frac{\partial}{\partial y}+\frac{lx}{2}\frac{\partial}{\partial
z},\;E_{3}=\frac{\partial}{\partial z},
\end{equation}
where $F=1+m(x^2+y^2)$.\\
Note that the map $\pi : M^3{(m,\;l)}\to (\r^2, h_1=\frac{dx^2+dy^2}{[1+m(x^2+y^2)]^2})$, $\pi(x,y,z)=(x, y)$ is a Riemannian submersion with totally geodesic fibres and $E_3=\frac{\partial}{\partial z}$ being tangent to the fibres.
 We assume that $\gamma:I\to(\r^2, h_1=\frac{dx^2+dy^2}{[1+m(x^2+y^2)]^2})$  with $\gamma(s)=(x(s),y(s))$ is
 an immersed regular curve parametrized by arclength and  $ \tilde{\gamma}: I \to (N^3,h)$ is a
horizontal lift of $\gamma$ with the geodesic curvatrue $\kappa$ and the torsion $\tau$. Therefore,  the surface $\Sigma=\bigcup_{s\in I}\pi^{-1}(\gamma(s))$ is  a Hopf cylinder in a BCV space $M^3{(m,\;l)}$. As  mentioned earlier,  the Hopf cylinder can be viewed as an isometric immersion $\varphi: \Sigma \to M^3{(m,\;l)}$ with the induced metric $g=\varphi^*h$, which can be  parametrized by $\varphi(s,z)=(\gamma(s),z)=(x(s),y(s),z)\subset M^3{(m,\;l)}$.\\

We would like to point out that if the curvatrue $\kappa$ is a nonzero constant,  then $4m-\kappa^2<0$. In fact,  clearly,  if $m\leq0$ then $4m-\kappa^2<0$. On the other hand, we see that if $m>0$, then $\gamma$ is a curve  on  a 2-sphere $S^2$  of Gauss curvature $4m$ . It follows from a
well-known fact in the differential geometry of curvers and surfaces that  radius of the curve $\frac{1}{|\kappa|}<\frac{1}{\sqrt{4m}}$, i.e., $4m-\kappa^2<0$.\\

We now give the  classification result of biharmonic conformal immersions of the Hopf cylinders $\Sigma$ into $M^3{(m,\;l)}$.
\begin{theorem}\label{hopf}
If a conformal immersion of a Hopf cylinder  into a BCV space \;$\varphi:(\Sigma,\bar{ g}=f^{-1}\varphi^*h)\to M^3{(m,\;l)}$  is proper biharmonic  with nonconstant conformal factor $f^{\frac{1}{2}}$ , then   $l=0$, and either (i): the cylinder has  nonzero constant mean curvature   $H=\frac{\kappa}{2}$ satisfying  $\kappa^2-4m>0$, and  $
f=d_1e^{z\sqrt{\kappa^2-4m} }+d_2e^{-z\sqrt{\kappa^2-4m} }$,
 where $d_1$, $d_2$ are constants with $d_1^2+d_2^2\neq0$ so that $f>0$, or (ii):  the cylinder has  nonconstant mean curvature   $H=\frac{\kappa}{2}$, and $f=\psi(z)\kappa^{-3/2}(s)=\psi(z,K)\kappa^{-3/2}(s,m,K)$, where $\kappa(s)=\kappa(s,m,K)$ and $\psi(z)=\psi(z,K)$ are given by (\ref{KK1}) and (\ref{KK2}), respectively, and $K$ is a constant.
\end{theorem}
\begin{proof}
Adopting the same notations as the above. It is clear to see that if $\kappa=0$, then the surface $\Sigma$ is minimal ( harmonic) and hence any conformal immersion of the surface is also harmonic.   From now on, we assume that $\kappa\neq0$. \\
 We know from  \cite{OW} that the  orthonormal frame $\{X=\frac{x'}{F}E_1+\frac{y'}{F}E_2,\;V=E_3,\;\xi=\frac{y'}{F}E_1-\frac{x'}{F}E_2\}$ is adapted to the Hopf cylinder $\Sigma=\cup_{s\in I}\pi^{-1}(\gamma(s))$, and also
${\rm Ric}^N(\xi,\xi)=4m-\frac{l^2}{2},\;({\rm Ric}(\xi))^{\top}=0$\;and $\tau=-\langle\nabla_XV,\xi\rangle=-\frac{l}{2}$. Combining these,  (\ref{Hopf}) turns into
\begin{equation}\label{HP}
\begin{cases}
\kappa''-\kappa^3-\frac{\kappa l^2}{2}+\kappa(4m-\frac{l^2}{2})+2\kappa' X(\ln f)\\+\kappa\{XX(\ln f)+VV(\ln f)+(X\ln f)^2+(V\ln f)^2\}=0,\\
3\kappa\kappa'+2\kappa^2X(\ln f)+\kappa l V(\ln f)=0,\\
l(\kappa'+\kappa X(\ln f))=0.
\end{cases}
\end{equation}

We  solve (\ref{HP}) by  considering the following two cases:\\

Case I: $l=0$.  Eqs. (\ref{HP}) becomes
\begin{equation}\label{hp2}
\begin{cases}
\kappa^2\frac{\partial \ln f}{\partial s}=-\frac{3}{2}\kappa'(s)\kappa,\\
\kappa''(s)-\kappa^3+4m\kappa+\kappa[\frac{\partial^2\ln f}{\partial s^2}+\frac{\partial^2\ln f}{\partial z^2}+(\frac{\partial\ln f}{\partial s})^2+(\frac{\partial\ln f}{\partial z})^2]+2\frac{\partial\ln f}{\partial s}\kappa'(s)=0.
\end{cases}
\end{equation}
We earlier remarked that that if  $\kappa$ is a nonzero constant,  then $4m-\kappa^2<0$. Then, 
a direct computation using (\ref{hp2}) gives $f=d_1e^{z\sqrt{\kappa^2-4m} }+d_2e^{-z\sqrt{\kappa^2-4m} }$,
 where $d_1$, $d_2$ are constants with $d_1^2+d_2^2\neq0$ so that $f>0$.

If  $\kappa$ is nonconstant,  then a straightforward computation using (\ref{hp2}) and simplifying the resulting equation yields
\begin{equation}\label{cy05}
\begin{array}{lll}
\ln f=-\frac{3}{2}\ln |\kappa(s)| +\ln \psi(z),\\
-2\kappa\kappa''(s)+3\kappa'^2(s)-4\kappa^4+16m\kappa^2+4\kappa^2(s)\frac{\psi''(z)}{\psi(z)}=0,
\end{array}
\end{equation}
where $\psi(z)$ is a positive function.
Since the second equation of (\ref{cy05}) implies that for any $s,z$, we can conclude 
\begin{equation}\label{cy6}
\begin{array}{lll}
3\kappa'^2(s)-2\kappa\kappa''(s)=4\kappa^2(\kappa^2-(4m+K))\;\;{\rm and}\;\;
\frac{\psi''(z)}{\psi(z)}=K,
\end{array}
\end{equation}
 where $K$ is a constant.
 
Solving  (\ref{cy6}), we obtain (see \cite{WQ})
\begin{equation}\label{KK1}
\kappa(s)=\kappa(s,m,K)=\begin{cases}
\frac{1}{C_1e^{2\sqrt{-(4m+K)}\;s}+D_1e^{-2\sqrt{-(4m+K)}\;s}+\sqrt{4C_1D_1+\frac{1}{4m+K}}},\;\;\;\;\;\;\;\;{\rm for}\; 4m+K<0,\\
\frac{4C_2}{16+C_2^2(s+D_2)^2},\;\;\;\;\;\;\;\;\;\;\;\;\;\;\;\;\;\;\;\;\;\;\;\;\;\;\;\;\;\;{\rm for}\; 4m+K=0,\\
\frac{1}{C_3\cos(2\sqrt{4m+K}\;s)+D_3\sin(2\sqrt{4m+K}\;s)+ \sqrt{\frac{1}{4m+K}+(C_3^2+D_3^2)}},\;{\rm for}\; 4m+K>0,
\end{cases}
\end{equation}
and

 \begin{equation}\label{KK2}
\psi(z)= \psi(z,K)=\begin{cases}
a_3\sin(\sqrt{-K}\;z)+b_3\cos(\sqrt{-K}\;z),\;\;\; {\rm for}\; K<0,\\
a_2z+b_2,\;\;\;\;\;\;\;\;\;\;\;\;\;\;\;\;\; {\rm for}\; K=0,\\
a_1e^{\sqrt{K}\;z}+b_1e^{-\sqrt{K}\;z},\;\;\;\;\;\;\;\; {\rm for}\; K>0,
\end{cases}
\end{equation}
 where $a_i$, $b_i$, $C_i$ and $D_i$ are constants, $i=1,2,3$,  so that $\kappa$ is a nonconstant positive function and $\psi>0$.
Clearly, $f=\psi(z)\kappa^{-3/2}(s)=\psi(z,K)\kappa^{-3/2}(s,m,K)$.\\

Case II: $l\neq0$.  We use the second  and the third equations of (\ref{HP}) to have $X(\ln f)=\frac{\partial\ln f}{\partial s}=-\frac{\kappa'}{\kappa}$ and $V(\ln f)=\frac{\partial\ln f}{\partial z}=-\frac{\kappa'}{l}$, which, together with that $f$ is nonconstant,   implies $\kappa=as+b$, where $a\neq0$ and $b$ are constants. Substituting these into the 1st equation of (\ref{HP}) and simplifying the resulting equation we have $-(as+b)^3+(4m-l^2+\frac{a^2}{l^2})(as+b)=0$. This implies $a=b=0$  contradicting $a\neq0$.\\
Summarizing both Case I and Case II, we complete the proof of the theorem.
\end{proof}

Note that when $l=0$,  the potential BCV space has to be $S^2\times\r$, $H^2\times\r$ or $\r^3$. As a consequence of Theorem \ref{hopf}, we have
\begin{corollary}\label{chopf}
(i) A proper biharmonic conformal immersion of a Hopf cylinder  into a BCV space \;$\varphi:(\Sigma,\bar{ g}=f^{-1}\varphi^*h)\to M^3{(m,\;l)}$  with nonconstant conformal factor $f^{\frac{1}{2}}$ 
exists only in the cases: $(\Sigma,\bar{ g}=f^{-1}\varphi^*h)\to S^2\times\r$, $(\Sigma,\bar{ g}=f^{-1}\varphi^*h)\to H^2\times\r$ or $(\Sigma,\bar{ g}=f^{-1}\varphi^*h)\to \r^3$. \\
(ii) There exist no proper biharmonic conformal immersions of the Hopf cylinders  into $S^3$, $Nil$, $\widetilde{SL}(2,\r)$ or $SU(2)$  with nonconstant conformal factor.
\end{corollary}
\begin{corollary}
(i) Any  proper biharmonic conformal immersion of a Hopf cylinder  into a BCV space exists only in the case where the target space can only be $S^2\times\r$, $H^2\times\r$, $\r^3$ or $SU(2)$.\\
(ii) There exists no proper biharmonic conformal immersion of a Hopf cylinder  into $S^3$, $\widetilde{SL}(2,\r)$ or $Nil$ no matter  what conformal factor is.

\end{corollary}
\begin{proof}
It follows from Theorem 2.2 in \cite{OW} that a proper Hopf cylinder in  3-dimensional BCV spaces exists only in the cases: $\Sigma$ in $S^2\times\r$ or $\Sigma$ in $SU(2)$.
Note that these maps are isometric immersions as a special subclass of conformal immersions for conformal factor $1$. From these and Corollary \ref{chopf},  the corollary follows.
\end{proof}

\begin{example}\label{hopff}
(1) Let  $\gamma:(0,\frac{\pi}{2\sqrt{2}})\to S^2=(\r^2, h_1=\frac{dx^2+dy^2}{[1+\frac{1}{4}(x^2+y^2)]^2})$, $\gamma(s)=(x(s),y(s))$ 
be an immersed regular curve parametrized by arclength with  the curvature $\kappa=\frac{1}{\frac{\sqrt{2}}{2}\sin(2\sqrt{2}\;s)+1}$. Then the Hopf cylinder $\varphi: (\Sigma,\varphi^*h) \to S^2\times\r=(\r^{3},h=\frac{dx^2+dy^2}{[1+\frac{1}{4}(x^2+y^2)]^2}+dz^2)$  with $\varphi(s,z)=(\gamma(s),z)$ has the mean curvatue $H= \frac{\kappa}{2}$. For $\psi=e^z$ and $f=\psi\kappa^{-\frac{3}{2}}=e^z\{\frac{\sqrt{2}}{2}\sin(2\sqrt{2}\;s)+1\}^{\frac{3}{2}}$, then the conformal immersion $\varphi:(\Sigma,\bar{ g}=f^{-1}\varphi^*h)\to S^2\times\r$  with  conformal factor $f^{\frac{1}{2}}$ is proper biharmonic.\\

Indeed, it is not difficult to see that taking $m=\frac{1}{4}$, $K=1$, $C_3=0$ and $D_3=\frac{\sqrt{2}}{2}$ in (\ref{KK1}), we get $\kappa=\frac{1}{\frac{\sqrt{2}}{2}\sin(2\sqrt{2}\;s)+1}<1$ and the  corresponding function $\psi=e^z$ by taking  $K=1$, $a_1=1$ and $b_1=0$ in (\ref{KK2}). 
It follows from a well-known fact in the differential geometry of curves and surfaces that  there exist the curve $\gamma:(0,\frac{\pi}{2\sqrt{2}})\to S^2$ with $\gamma(s)=(x(s),y(s))$ such that its curvature $\kappa=\frac{1}{\frac{\sqrt{2}}{2}\sin(2\sqrt{2}\;s)+1}$. This implies that
the Hopf cylinder $\Sigma=\bigcup_{s\in I}\pi^{-1}(\gamma(s))$  in $S^2\times\r$ has the mean curvatue $H= \frac{\kappa}{2}$,
which can be viewed as an isometric immersion $\varphi: (\Sigma, g=\varphi^*h)\to (\r^2,h=\frac{dx^2+dy^2}{[1+\frac{1}{4}(x^2+y^2)]^2}+dz^2)$  with $\varphi(s,z)=(\gamma(s),z)=(x(s),y(s),z)\subset S^2\times\r$. Thus,  by Theorem \ref{hopf}, we obtain the statement (1).\\

In a similar way, we have the following \\
(2) Let  $\gamma:I\to H^2=(\r^{2},h_1=\frac{dx^2+dy^2}{[1-\frac{1}{4}(x^2+y^2)]^2})$, $\gamma(s)=(x(s),y(s))$ 
be an immersed regular curve parametrized by arclength with  the curvature $\kappa=\frac{1}{1+s^2}$. 
For  $f=(1+s^2)^{\frac{3}{2}}e^z$,
then the conformal immersion of the Hopf cylinder \;$\varphi:(\Sigma,\bar{ g}=f^{-1}\varphi^*h)\to H^2\times\r$ with $\varphi(s,z)=(\gamma(s),z)$   is proper biharmonic. Note that the associate surface  has the mean curvature $H=\frac{1}{2(1+s^2)}$.\\

(3)  Consider the curve  $\gamma:I\to (\r^{2},dx^2+dy^2)$ with $\gamma(s)=(\ln(\sqrt{1+s^2}+s)+\ln 4,
 \sqrt{1+s^2})$. Then the curve $\gamma$ has the curvature $\kappa=\frac{1}{1+s^2}$. For  $f=(1+s^2)^{\frac{3}{2}}z$ with  $z>0$,  then the conformal immersion of the Hopf cylinder into $\r^2\times\r^{+}$, i.e.,  $\varphi:(\Sigma, f^{-1}(ds^2+dz^2))\to (\r^2\times\r^{+},h=dx^2+dy^2+dz^2)$ with $\varphi(s,z)=(\ln(\sqrt{1+s^2}+s)+\ln 4, \sqrt{1+s^2},\; z)$,  is a proper biharmonic map.  Note that the associate surface  has the mean curvature $H=\frac{1}{2(1+s^2)}$.
\end{example}
\begin{remark}
Note that  biharmonic conformal immersions of  the cylinders into $\r^3$ and of constant mean curvature Hopf cylinders into $S^2\times \r$ have been characterized in \cite {WQ} and \cite{Ou4}, respectively.  Our Theorem \ref{hopf} recovers the above results. Interestingly,  we give an  example of proper biharmonic conformal immersion of  a nonconstant mean curvature  Hopf cylinder into $S^2\times \r$. 

\end{remark}

Statements and Declarations:\\
1. Funding: Ze-Ping Wang was supported by the Natural Science Foundation of China (No. 11861022) and by the Scientific and Technological Project in Guizhou Province ( Grant no. Qiankehe LH [2017]7342).\\
2. Conflict of interest: The authors declare that they have no conflict of interest.\\
3. Data Availability Statement: The results/data/figures in this manuscript have not been published elsewhere, nor are they under consideration by another publisher.

\end{document}